\documentclass[a4paper,12pt]{amsart}

\usepackage{enumerate,layout,amssymb,epsf,mathrsfs,braket}
\usepackage{eepic,graphicx}
\usepackage{wrapfig}

\newtheorem{theorem}{Theorem}[section]

\newtheorem{proposition}[theorem]{Proposition}

\newtheorem{lemma}[theorem]{Lemma}

\theoremstyle{definition}

\newtheorem{definition}[theorem]{Definition}

\newtheorem{remark}[theorem]{Remark}

\newcommand{\Z}{{\mathbb Z}}
\newcommand{\N}{{\mathbb N}}

\newcommand{\R}{{\mathbb R}}

\renewcommand{\L}{{\mathcal L}}

\DeclareMathOperator{\pref}{p}
\DeclareMathOperator{\suf}{s}
\DeclareMathOperator{\spref}{sp}
\DeclareMathOperator{\ssuf}{ss}
\DeclareMathOperator{\comp}{vec}
\DeclareMathOperator{\Pref}{Pref}
\DeclareMathOperator{\Suf}{Suf}
\numberwithin{equation}{section}

\date{\today}

\begin{document}

\title[the unilateral recognizability]%
{On B.~Moss\'e's unilateral recognizability theorem}

\author[S. Akiyama]{Shigeki Akiyama}
\address{Institute of Mathematics, University of Tsukuba, Tennodai 1-1-1,
Tsukuba, Ibaraki, 305-8571 JAPAN.}
\email{akiyama@math.tsukuba.ac.jp}

\author[B. Tan]{Bo Tan}
\address{School of Mathematics and Statistics, Huazhong University of Science \& Technology, Wuhan 
430074, P.R.CHINA.}
\email{tanbo@hust.edu.cn}

\author[H. Yuasa]{Hisatoshi Yuasa}
\address{Division of Science, Mathematics and Information, Osaka Kyoiku University,
4-698-1 Asahigaoka, Kashiwara, Osaka 582-8582, JAPAN.}
\email{hyuasa@cc.osaka-kyoiku.ac.jp}

\begin{abstract}
We complete statement and proof for B.~Moss\'e's unilateral recognizability theorem. We also 
provide an algorithm for deciding the unilateral non-recognizability of a given primitive substitution. 
\end{abstract}



\maketitle

\section{Introduction}\label{intro}

Let $A$ be a finite alphabet consisting of at least two letters. Let $A^+$ denote the set of nonempty 
words over the alphabet $A$. Every map $\sigma$ from the alphabet $A$ to $A^+$ is called a 
{\em substitution} on the alphabet $A$. The substitution $\sigma$ is said to be {\em primitive} if 
there exists $k \in \N=\{1,2,\dots\}$ such that for any pair $(a,b) \in A \times A$, the letter 
$a$ occurs in the word $\sigma^k(b)$. Throughout the present paper, a given substitution is assumed 
to be primitive. Suppose that the substitution $\sigma$ has a fixed point 
$u=u_0u_1u_2 \dots$ in $A^{\Z_+}$, where $\Z_+=\Set{0} \cup \N$. If the fixed point $u$ is aperiodic 
under the left shift $T$ on $A^{\Z_+}$, i.e.\ $T^iu=u$ for all $i \in \N$, then the substitution 
$\sigma$ is said to be {\em aperiodic}. There is an algorithm \cite{HL,pansiot2} which can check 
whether a given substitution is aperiodic. We always assume that the substitution $\sigma$ is aperiodic. 
For every $p \in \N$, set 
\[
E_p=\Set{0} \cup \Set{\left\vert \sigma^p(u_{[0,n)}) \right\vert | n \in \N},
\]
where $|w|$ is the length of a word $w$. The elements of $E_p$ are called {\em natural 
$p$-cutting points}; see also \cite[\S~3]{MR1168468}, \cite[\S~3.4]{MR1709427} and 
\cite[\S~7.2.1]{MR1970385}. It is clear that $E_q \subsetneq E_p$ whenever $q > p$. 
The substitution $\sigma$ is said to be {\em unilaterally recognizable} 
\cite[p.~530]{MR873430} if there exists $L \in \N$ such that if 
$u_{[i,i+L)}=u_{[j,j+L)}$ and $i \in E_1$ then $j \in E_1$. This definition does not depend on the 
choice of the fixed point $u$ of the substitution $\sigma$. Also, the substitution $\sigma$ is said 
to be {\em bilaterally recognizable} \cite[D\'efinition~1.2]{MR1168468} if there exists 
$L \in \N$ such that if $u_{[i-L,i+L)}=u_{[j-L,j+L)}$ and $i \in E_1$ then $j \in E_1$. 

The unilateral recognizability is an important notion from viewpoints of subshifts arising from 
substitutions. If the substitution $\sigma$ is unilaterally recognizable, 
then a unilateral subshift $X_\sigma$ arising from the substitution $\sigma$ has a Kakutani-Rohlin 
partition \cite{HPS} built on a {\em clopen} subset $\sigma(X_\sigma)$ of $X_\sigma$. 
Proposition VI.~6 of \cite{MR924156} states that given a point $x=x_0x_1x_2 \ldots \in X_\sigma$, the 
first return time of the point $\sigma(x)$ to the clopen subset $\sigma(X_\sigma)$ equals 
$|\sigma(x_0)|$. This leads to a fact that the first return map on 
$\sigma(X_\sigma)$ is a topological factor of $X_\sigma$, which shows a self-similarity of 
$X_\sigma$ if the substitution $\sigma$ is injective on the alphabet $A$; see 
\cite[Corollary~VI.~8]{MR924156}. 
It is also a significant consequence of the unilateral recognizability that 
$\sigma(X_\sigma)$ is open; see \cite[Proposition VI.~3]{MR924156} and \cite[Lemme~2]{MR873430}. 
The unilateral recognizability is a premise of the celebrated theorem of \cite{MR873430}, which 
characterizes eigenvalues and eigenfunctions of the subshift $X_\sigma$. 

B.~Moss\'e gave \cite[Th\'eor\`eme~3.1]{MR1168468} to characterize the unilateral 
{\em non}-recognizability. However, if we dare say, it is incomplete and should be 
formulated as follows. 
\begin{theorem}\label{necsuffcondunilrecog}
The following are equivalent$:$
\begin{enumerate}
\item\label{non-recog}
the substitution $\sigma$ is not unilaterally recognizable$;$
\item\label{charac}
for each $L \in \N$, there exist $i,j \in \Z_+$ such that 
\begin{itemize}
\item
$\sigma(u_j)$ is a strict suffix of $\sigma(u_i);$
\item
$\sigma(u_{i+k})=\sigma(u_{j+k})$ for each integer $k$ with $1 \le k \le L$.
\end{itemize}
\end{enumerate}
\end{theorem}
Recall that the substitution $\sigma$ is assumed to be aperiodic. 
The word $B$ appearing in the statement of \cite[Th\'eor\`eme~3.1]{MR1168468} corresponds to a factor 
of $\sigma(u_iu_{i+1} \dots u_{i+L})$. The letters $a$ and $b$ in the statement correspond to $u_i$ and 
$u_j$, respectively. It is important to regard $u_i$ and $u_j$ as letters accompanied with information 
on the positions $i$ and $j$ where they occur. 

B.~Moss\'e's proof for her characterization would be difficult to completely 
follow, in particular, Part~(4) in p.~332. No proofs for it can be found in recent textbooks 
\cite{MR1970385,MR2041676,MR2590264}, though a proof for the {\em bilateral} recognizability 
is written in \cite[pp.~163-164]{MR2041676}. However, the difficulty is overcome in Step~3 in the 
proof of Theorem~\ref{necsuffcondunilrecog}, which is one of the goals of the present paper. 
In Lemma~\ref{bil_recog_index}, we also show that an index under which the aperiodic substitution 
$\sigma$ is bilaterally recognizable can be described in terms of only parameters derived from the 
substitution $\sigma$ itself. In fact, the statement of the lemma excluding the computability of the 
index is exactly \cite[Th\'eor\`eme~3.1 bis.]{MR1168468}, which is presented in terms of 
{\em local unique composition property} defined by Property~(6) in \cite{Sol}. As a consequence of 
the lemma, Proposition~\ref{constpiscomputable} affirms that a constant $p \in \N$ for which 
if $\sigma^{p-1}(a) \ne \sigma^{p-1}(b)$ and $a, b \in A$ then $\sigma^k(a) \ne \sigma^k(b)$ 
for all $k \in \Z_+$.

The other goal is to present an algorithm which determines whether or not a given aperiodic, 
primitive substitution is unilaterally recognizable. The algorithm is described in terms of 
the existence of a cycle in a directed, finite graph whose vertex set consists of a pair of those 
words of constant length which occur in the sequence $u$. 
In view of Theorem~\ref{necsuffcondunilrecog}, it may be of interest to find a computable constant 
$M$, such that the existence of a word $v=v_1v_2 \dots v_{M+1}$ of length $M+1$ satisfying that 
\begin{itemize}
\item
$\sigma(v_1)$ is not a strict suffix of $\sigma(w_1)$, or 
\item
$\sigma(v_k) \ne \sigma(w_k)$ for some integer $k$ with $2 \le k \le M+1$
\end{itemize}
is equivalent to the unilateral recognizability of the substitution $\sigma$, 
which would give an easier algorithm.

\section{$K$-power free sequences}\label{sec_freeness}

We shall make terminology excepting that done in the preceding section. The {\em empty word} is denoted 
by $\Lambda$. Set $A^\ast=A^+ \cup \Set{\Lambda}$. We say that a word $w \in A^\ast$ {\em occurs} in a 
word $v \in A^\ast$ if there exist $p,s \in A^\ast$ such that $v=pws$. 
We then write $w \prec v$. More specifically, $w$ is said to occurs at the 
position $|p|+1$ in $v$. The position is called an occurrence of $w$ in $v$. Let $|v|_w$ 
denote the number of occurrences of $w$ in $v$. 
The words $p$ and $s$ are called a {\em prefix} and {\em suffix} of $v$, 
respectively. We then write $p \prec_{\pref} v$ and $s \prec_{\suf} v$, respectively. If 
$|p|<|v|$ (resp.\ $|s|<|v|$), then $p$ (resp.\ $s$) is called a strict prefix (resp.\ suffix) 
of $v$, and then we write $p \prec_{\spref} v$ (resp.\ $s \prec_{\ssuf} v$). We can also define 
$e \prec_{\suf} f$ for $e,f \in A^{\Z_+}$ in such a way that $e=f_{[n,+\infty)}$ for some $n \in \Z_+$. 
A power of a word $w \in A^\ast$ is a word of the form 
$\underbrace{ww \dots w}_{n \textrm{ times}}$ with some $n \in \Z_+$. The power is denoted by $w^n$. 
In particular, $w^0 = \Lambda$. Set 
\[
I_i = \min_{a \in A}|\sigma^i(a)| \textrm{ and } S_i = \max_{a \in A}|\sigma^i(a)|.
\]

A nonnegative square matrix $M$ is said to be {\em primitive} if there exists $k \in \N$ for which 
$M^k$ is positive. The {\em incidence matrix} $M_\sigma$ of the substitution $\sigma$ is defined to 
be an $A \times A$ matrix whose $(a,b)$-entry equals $|\sigma(a)|_b$. The matrix $M_\sigma$ is 
primitive and has a positive, right eigenvector $\beta=(\beta_a)_{a \in A}$ corresponding to Perron 
eigenvalue $\lambda$ of $M_\sigma$, i.e.\ the absolute value of any other eigenvalue is less than 
$\lambda$. See for example \cite[Sections~4.2-4.5]{MR1369092}. Since for all $a \in A$,
\[
\sum_{b \in A}({M_\sigma}^n)_{a,b} \beta_b= \lambda^n \beta_a,
\]
it follows that for all $a \in A$ and $n \in \N$,
\begin{equation*}\label{originalsub}
\frac{\min_{b \in A}\beta_b}{\max_{b \in A}\beta_b} \cdot \lambda^n \le \sum_{b \in A}
({M_\sigma}^n)_{a,b} \le \frac{\max_{b \in A}\beta_b}{\min_{b \in A}\beta_b} \cdot \lambda^n.
\end{equation*}
Put 
\[
C = \left\lceil \frac{\max_{b \in A}\beta_b}{\min_{b \in A}\beta_b} \right\rceil.
\]
It follows that for all $a \in A$ and $n \in \N$,
\begin{equation}\label{conseqPF}
C^{-1}\lambda^n \le |\sigma^n(a)| \le C\lambda^n.
\end{equation}

Given a sequence $v \in A^{\Z_+}$, set 
\begin{align*}
\L(v)^+ &= \Set{v_{[i,j]}:=v_iv_{i+1} \dots v_j|i,j \in \Z_+,i \le j}; \\
\L(v) &= \L(v)^+ \cup \Set{\Lambda}; \\
\L_k(v) &= \Set{w \in \L(v)| \vert w\vert =k}.
\end{align*}

We say that a word $w \in A^\ast$ {\em occurs} at a position $i \in \Z_+$ in $v$ if 
$v_{[i,i+|w|)}=w$. The integer $i$ is called an {\em occurrence} of the word $w$. The fixed point $u$ 
of the substitution $\sigma$ is {\em uniformly recurrent}, i.e.\ given a word $w \in \L(u)$, there 
exists $g \in \N$ so that any interval of length $g$, which is a subset of $\Z_+$, includes an 
occurrence of the word $w$. For, the primitivity of the substitution $\sigma$ implies the existence 
of $n \in \N$ such that $w \prec \sigma^n(a)$ for all $a \in A$. Then, the length $g$ can be chosen 
to be $2\max_{a \in A}|\sigma^n(a)|$, because 
\[
u=\sigma^n(u)=\sigma^n(u_0)\sigma^n(u_1) \dots
\] 
We shall refer to the length $g$ as a {\em gap} of occurrences of the word $w$ in the sequence $u$. 
Let $g$ be the maximal value of gaps of occurrences of words belonging to $\L_2(u)$. 
\begin{lemma}
The value $g$ is computable.
\end{lemma}
\begin{proof}
Consider an auxiliary substitution $\sigma_2:\L_2(u) \to \L_2(u)^+$~\cite[pp.~95-96]{MR924156}, 
where $\L_2(u)$ is regarded as a finite alphabet. The substitution $\sigma_2$ is defined by for 
$w \in \L_2(u)$,
\[
\sigma_2(w)=\sigma(w)_{[1,2]} \sigma(w)_{[2,3]} \dots \sigma(w)_{[|\sigma(w_1)|,|\sigma(w_1)|+1]}.
\]
The substitution $\sigma_2$ is primitive~\cite[Lemma~V.12]{MR924156}. Observe that $\# \L_2(u) \le 
(\# A)^2$ and $|\sigma_2(w)|=|\sigma(w_1)|$ for all $w \in \L_2(u)$. 
Set 
\[
u^{(2)}=(u_0u_1)(u_1u_2)(u_2u_3) \dots,
\]
which is a fixed point of $\sigma_2$ in $\L_2(u)^{\Z_+}$~\cite[Lemma~V.11]{MR924156}. Observe that 
given $w \in \L_2(u)$ and $i \in \Z_+$, $i$ is an occurrence of $w$ in $u$ if and only if $u^{(2)}_i=w$. 
In view of \cite[Theorem~2.9]{Sen}, the least $n \in \N$ for which every entry of $(M_{\sigma_2})^n$ 
is positive has a upper bound $(\# \L_2(u))^2-2 \cdot \# \L_2(u) + 2 \le (\# A)^4 -2(\# A)^2 + 2$. 
Put 
\[
n_0=(\# A)^4 -2(\# A)^2 + 2.
\]
As done in the proof of \cite[Lemma~5.1~(iii)]{Y4}, define a 
$\L_2(u) \times A$-matrix $N$ by letting $N_{w,a}=(M_\sigma)_{w_1,a}$ for all 
$(w,a) \in \L_2(u) \times A$. Let $v$ be a positive eigenvector of $M_\sigma$ corresponding to 
$\lambda$, as above. Since $M_{\sigma_2}N=NM_\sigma$, whose $(w,a)$-entry is $|\sigma^2(w_1)|_a$ for 
all $(w,a) \in \L_2(u) \times A$, a positive vector $Nv$ is an eigenvector of $M_{\sigma_2}$ 
corresponding to its eigenvalue $\lambda$. Consequently, the number $\lambda$ is a dominant eigenvalue 
of a primitive matrix $M_{\sigma_2}$; see for example \cite[p.~108 and Theorem~4.5.11]{MR1369092}. 
We then obtain that for all $w \in \L_2(u)$ and $n \in \N$,
\[
\left(C \max_{a \in A}|\sigma(a)|\right)^{-1} \lambda^n \le |{\sigma_2}^n(w)| \le 
C \max_{a \in A}|\sigma(a)| \lambda^n.
\]
It follows finally that the value $g$ has a upper bound $2C \max_{a \in A}|\sigma(a)| \lambda^{n_0}$. 
\end{proof}

A word $v \in A^+$ is said to be {\em primitive} \cite[D\'efinition~2.2]{MR1168468} if it holds that 
\[
v = w^n, w \in A^+, n \in \N \Rightarrow w = v. 
\]
We say that a sequence $v=v_0v_1v_2 \ldots \in A^{\Z_+}$ is {\em ultimately periodic} if there exist 
$n \in \Z_+$ and word $w \in A^+$ for which 
\[
v=v_{[0,n)}www \dots
\]
If an ultimately periodic sequence $v$ is uniformly recurrent, then 
$v$ is {\em periodic}, i.e.\ $v$ is written as an infinite repetition of a single word. 
Recall that the substitution $\sigma$ is assumed to be aperiodic. 
\begin{lemma}[{\cite[Lemme 2.5]{MR1168468}}]\label{suffcondperiodic}
There does not exist $N,p \in \N$ and primitive word $v \in A^+$ for which
\begin{itemize}
\item
$\sigma^p(w) \prec v^N$ for any $w \in \L_2(u);$
\item
$2|v| \le \min_{a \in A}|\sigma^p(a)|$. 
\end{itemize}
\end{lemma}

\begin{proof}
For the sake of completeness, we give a proof. Of course, the idea is due to 
B.~Moss\'e~\cite{MR1168468}. Assume that there exists such a triple $N,p$ and $v$. Since 
$\sigma^p(u_i) \prec v^N$, there exist words $\alpha_i \prec_{\ssuf} v, \beta_i \prec_{\spref} v$ and 
$n_i \in \Z_+$ for which $\sigma^p(u_i)=\alpha_i v^{n_i} \beta_i$. If $n_i=0$, then 
$|\sigma^p(u_i)| < 2|v|$, which contradicts the hypothesis. Since 
\[
\sigma^p(u_iu_{i+1}) = \alpha_iv^{n_i}\beta_i \alpha_{i+1}v^{n_{i+1}}\beta_{i+1} \prec v^N,
\]
we see that $v\beta_i \alpha_{i+1}v \prec v^N$. This implies that $\beta_i \alpha_{i+1}$ is a power 
of $v$; see \cite[Propri\'et\'e~2.3]{MR1168468}. Hence, $u=\sigma^p(u)=\alpha_0 vvv \dots$ This is 
a contradiction as $\sigma$ is aperiodic. 
\end{proof}

\begin{lemma}[{\cite[Th\'eor\`eme~2.4]{MR1168468}}]\label{uppbndforpow}
If $n \in \N$ and $w^n \in \L(u)^+$, then $n < 2\lambda(g+1)C^2$.
\end{lemma}
\begin{proof}
For the sake of completeness, we present a proof. Again, the idea is due to B.~Moss\'e~\cite{MR1168468}. 
Suppose that $w$ is a primitive word 
and $w^n \in \L(u)$ for some $n \in \N$. There exists $p \in \N$ for which 
\[
\frac{1}{2}\min_{a \in A}|\sigma^{p-1}(a)| \le |w| < \frac{1}{2}\min_{a \in A}|\sigma^p(a)|.
\]
Recall that $v \prec u_{[i,i+g)}$ for all words $v \in \L_2(u)$ and $i \in \Z_+$. Since 
$2|w| < \min_{a \in A}|\sigma^p(a)|$ and $\sigma$ is aperiodic, it follows from 
Lemma~\ref{suffcondperiodic} that 
\[
n|w|=|w^n| < (g+1)\max_{a \in A}|\sigma^p(a)|.
\]
Hence,
\[
n < \frac{(g+1)\max_{a \in A}|\sigma^p(a)|}{\frac{1}{2}\min_{a \in A}|\sigma^{p-1}(a)|} \le 2(g+1)C^2
\lambda.
\]
The last inequality follows from \eqref{conseqPF}. 
\end{proof}

\begin{lemma}\label{comlexity}
$\sharp \L_n(u) \le \lambda C^2 (\# A)^2 n$ for every $n \in \N$.
\end{lemma}
\begin{proof}
This proof follows that of \cite[Proposition V.19]{MR924156}. Fix $n \in \N$. Find $p \in \N$ 
so that $\min_{a \in A}|\sigma^{p-1}(a)| \le n \le \min_{a \in A}|\sigma^p(a)|$. Then, 
\[
\# \L_n(u) \le (\# A)^2 \min_{a \in A}|\sigma^p(a)| \le (\# A)^2 \frac{\min_{a \in A}|\sigma^p(a)|}
{\min_{a \in A}|\sigma^{p-1}(a)|} n \le \lambda C^2 (\# A)^2 n. 
\]
This completes the proof. 
\end{proof}
See also \cite{ELR1,pansiot1} and \cite[Theorem 24]{MR1709427}. Set 
\[
K= \left\lceil \lambda C^2 \max \Set{ 2(g+1), (\# A)^2} \right\rceil. 
\]
Then, the fixed point $u$ of the substitution $\sigma$ is {\em $K$-power free}, in other words, it 
holds that if $v^N \in \L(u)$ and $N \ge K$ then $v = \Lambda$. Hence, the constant $K$ is inevitably 
greater than or equal to two. In general, if a uniformly recurrent 
sequence $v \in A^{\Z_+}$ is aperiodic, then given a word $w \in \L(v)^+$ there exists a constant 
$L \in \N$ for which $w^L \notin \L(v)$. However, the constant $L$ may depend on the choice of 
the word $w$. 

\section{B.~Moss\'e's characterization of the unilateral non-recognizability}

\begin{definition}\label{naturalcuttings}
\begin{enumerate}
\item\label{natural-p-cutting}
A finite sequence: 
\[
\Set{\alpha,\sigma^p(u_{i^\prime}),\sigma^p(u_{i^\prime+1}),\dots,
\sigma^p(u_{i^\prime+k-1}),\beta}
\]
of words over the alphabet $A$ is called a {\em natural $p$-cutting} of $u_{[i,i+\ell)}$ if 
\begin{itemize}
\item
$\alpha \prec_{\suf} \sigma^p(u_{i^\prime-1})$;
\item
$\beta \prec_{\pref} \sigma^p(u_{i^\prime+k})$;
\item 
$u_{[i,i+\ell)}=\alpha \sigma^p(u_{i^\prime}) \sigma^p(u_{i^\prime+1}) \dots 
\sigma^p(u_{i^\prime+k-1})\beta$, where $i$;
\item
$i+|\alpha| = |\sigma^p(u_{[0,i^\prime)})|$.
\end{itemize}
\item\label{awordhasthesamecutting}
If a word $w$ occurs at positions $i$ and $j$ in the sequence $u$, then the word $w$ is said to have 
the {\em same natural $p$-cutting} at the positions $i$ and $j$ if 
\[
\left(E_p \cap [i,i+|w|)\right) + (j-i) = E_p \cap [j,j+|w|),
\]
where $E+i=\Set{e+i|e \in E}$ if $E$ is a finite subset of $\Z_+$ and $i \in \Z_+$. 
\end{enumerate}
\end{definition}

Compare these definitions with the original ones in \cite[\S~3]{MR1168468}. 
We do not exclude the possibility that 
\[
\alpha = \sigma^p(u_{i^\prime-1}), \alpha = \Lambda, \beta = \sigma^p(u_{i^\prime+k}) 
\textrm{ or } \beta = \Lambda.
\]
Since we always require $k \ge 1$ in Definition~\ref{naturalcuttings}~\eqref{natural-p-cutting}, 
not every $u_{[i,i+\ell)}$ has a natural $p$-cutting. It is not necessary that a natural $p$-cutting 
is uniquely determined for given $i$ and $\ell$ in 
Definition~\ref{naturalcuttings}~\eqref{natural-p-cutting}. 

\renewcommand{\proofname}{Proof of Theorem~{\em \ref{necsuffcondunilrecog}}}
\begin{proof}

Put 
\begin{equation}\label{def_of_k}
k=C^4(K+1)+2C^2+1.
\end{equation}
To show the implication \eqref{non-recog} $\Rightarrow$ \eqref{charac}, 
assume that the substitution $\sigma$ is not unilaterally recognizable. 

{\bf Step~1.} It follows from  Lemma~\ref{1st_Step} below that for each $p \in \N$, there 
exist integers $i_p \in E_1$, $j_p \notin E_1$, $i_p^\prime,j_p^\prime \ge 0$, 
$h_p,\ell_p \ge 1$ and words $\alpha_p,\gamma_p^\prime \in A^\ast$, $\gamma_p \in A^+$ 
such that 
\begin{itemize}
\item
$u_{[i_p,i_p+\ell_p)} =u_{[j_p,j_p+\ell_p)}$;
\item
$u_{[i_p,i_p+\ell_p)}$ has a natural $p$-cutting: 
\[
\Set{\alpha_p,\sigma^p(u_{i_p^\prime}),\sigma^p(u_{i_p^\prime+1}),\dots,
\sigma^p(u_{i_p^\prime+k-1})};
\] 
\item
$u_{[j_p+|\alpha_p|,j_p+\ell_p)}$ has a natural $p$-cutting:
\[
\Set{\gamma_p,\sigma^p(u_{j_p^\prime}),\sigma^p(u_{j_p^\prime+1}),\dots,
\sigma^p(u_{j_p^\prime+h_p-1}),\gamma_p^\prime}.
\]
\end{itemize}
Set 
\[
m_p = \min \Set{m \in \N|\alpha_p \gamma_p \prec_{\suf} 
\sigma^p(u_{[j_p^\prime-m,j_p^\prime)})}.
\]
Since 
\[
(m_p-1)I_p \le |\sigma^p(u_{[j_p^\prime-m_p+1,j_p^\prime)})| < |\alpha_p \gamma_p| 
\le 2S_p,
\]
we obtain that for all $p \in \N$, 
\[
m_p < 2C^2+1.
\]
Since
\[
h_p I_p \le |\gamma_p \sigma^p(u_{[j_p^\prime,j_p^\prime+h_p)})\gamma_p^\prime| 
 = |\sigma^p(u_{[i_p^\prime,i_p^\prime+k)}| \le k S_p
\]
and
\[(h_p + 2)S_p \ge |\gamma_p \sigma^p(u_{[j_p^\prime,j_p^\prime+h_p)})\gamma_p^\prime| 
= |\sigma^p(u_{[i_p^\prime,i_p^\prime+k)}| \ge k I_p,
\]
we obtain that for all $p \in \N$,
\[
kC^{-2}-2 \le h_p \le kC^2.
\]
It follows that a set:
\[
\Set{(m_p,h_p,u_{[i^\prime_p-1,i^\prime_p+k)},u_{[j_p^\prime-m_p,j_p^\prime+h_p]})|p \in \N}
\]
has a finite cardinality. Hence, the pigeonhole principle implies that for some infinite set 
$I \subset \N$, a set:
\[
\Set{(m_p,h_p,u_{[i^\prime_p-1,i^\prime_p+k)},u_{[j_p^\prime-m_p,j_p^\prime+h_p]})|\ p \in I} 
\]
is a singleton. It allows us to put $m = m_p$ and $h = h_p$ for any $p \in I$. 

{\bf Step~2.}
Let $p,q \in I$ with $p < q$ be arbitrary. We have two natural $q$-cuttings: 
\begin{equation}\label{nat_q_cut_1}
\Set{\gamma_q,\sigma^q(u_{j_q^\prime}),\sigma^q(u_{j_q^\prime+1}), \dots, 
\sigma^q(u_{j_q^\prime+h-1}),\gamma_q^\prime}
\end{equation}
of a word occurring at the position $j_q+|\alpha_q|$ and 
\begin{equation}\label{nat_q_cut_2}
\Set{\sigma^{q-p}(\gamma_p),\sigma^q(u_{j_q^\prime}),\sigma^q(u_{j_q^\prime+1}), \dots, 
\sigma^q(u_{j_q^\prime+h-1}),\sigma^{q-p}(\gamma_p^\prime)}
\end{equation}
of a word occurring at the position $j_q+|\alpha_q \gamma_q|-|\sigma^{q-p}(\gamma_p)|$. It would be 
worthwhile observing that $\sigma^{q-p}(\gamma_p) \prec_{\suf} \sigma^q(u_{j_q^\prime-1})$ and 
$\sigma^{q-p}(\gamma_p^\prime) \prec_{\pref} \sigma^q(u_{j_q^\prime+h_q})$. 
Assume that the natural $q$-cuttings \eqref{nat_q_cut_1} and \eqref{nat_q_cut_2} are different. Then, 
one of the inequalities $|\gamma_q| \ne |\sigma^{q-p}(\gamma_p)|$ and 
$|\gamma_q^\prime| \ne |\sigma^{q-p}(\gamma_p^\prime)|$ follows.

Consider the case $|\gamma_q| > |\sigma^{q-p}(\gamma_p)|$. 
Since 
\begin{align*}
\gamma_q\sigma^q(u_{[j_q^\prime,j_q^\prime+h)}) \gamma_q^\prime &=
\sigma^q(u_{[i_q^\prime,i_q^\prime+k)}) \\ 
&= \sigma^{q-p}(\sigma^p(u_{[i_p^\prime,i_p^\prime+k)})) \\
& = \sigma^{q-p}(\gamma_p \sigma^p(u_{[j_p^\prime,j_p^\prime+h)}) \gamma_p^\prime) \\
& = \sigma^{q-p}(\gamma_p)\sigma^q(u_{[j_q^\prime,j_q^\prime+h)})\sigma^{q-p}(\gamma_p^\prime),
\end{align*}
a power $v^N$ of a nonempty word $v \prec_{\ssuf} \gamma_q$ occurs in 
$\sigma^q(u_{[j_q^\prime,j_q^\prime+h)})$ as a prefix. 
By using the fact that $v \prec_{\suf} \sigma^q(u_{j_q^\prime-1})$, we can see that 
\begin{align}
\max \Set{N \in \N| v^N \prec_{\pref} \sigma^q(u_{[j_q^\prime,j_q^\prime+h)}) } & \ge 
\frac{h I_q}{S_q} -1 \label{estimate_max_power} \\ 
& \ge (kC^{-2}-2)C^{-2}-1 \notag \\
& = K + C^{-4} \notag \\
& > K, \notag 
\end{align}
where the equality follows from \eqref{def_of_k}. This contradicts the $(K+1)$-power freeness of the 
sequence $u$, i.e.\ Lemma~\ref{uppbndforpow}. The same contradiction 
emerging in the other cases, we conclude that for any $p,q \in I$ with $p < q$, 
\[
\gamma_q=\sigma^{q-p}(\gamma_p).
\]

{\bf Step~3.} 
Choose integers $p < q$ in $I$ so that 
\[
|\sigma^{q-p-1}(\gamma_p)| \ge L.
\]
Observe how $u_{[i_q^\prime-1,i_q^\prime+k)}$ goes to 
$\sigma^q(u_{[i_q^\prime-1,i_q^\prime+k)})$ via 
$\sigma^p(u_{[i_q^\prime-1,i_q^\prime+k)})$; see~Figure~\ref{pic_i_0-p-q}. Since 
$\gamma_p \prec_{\pref} \sigma^p(u_{[i_q^\prime,i_q^\prime+k)})$, $\gamma_q \prec_{\pref} 
\sigma^q(u_{[i_q^\prime,i_q^\prime+k)})$ and $\sigma^{q-p}(\gamma_p) = \gamma_q$, we can see 
that 
$u_{[i_q+|\alpha_q|,i_q+|\alpha_q\gamma_q|)}=\gamma_q$ has a natural $1$-cutting: 
\begin{equation}\label{ntrl_1_cttng_iq_gammaq}
\Set{\sigma(u_{i^{\prime \prime}}),\sigma(u_{i^{\prime \prime}+1}), \dots, 
\sigma(u_{i^{\prime \prime}+ \vert \sigma^{q-p-1}(\gamma_p) \vert -1})},
\end{equation}
where $i^{\prime \prime}=\left|\sigma^{q-1}\left(u_{\left[0,i_q^\prime\right)}\right)\right|$. 
Remark that
\begin{equation}\label{ancestor_of_gammaq_at_iq}
u_{\left[i^{\prime \prime},i^{\prime \prime}+\left|\sigma^{q-p-1}(\gamma_p)\right|\right)}
=\sigma^{q-p-1}(\gamma_p).
\end{equation}
\begin{figure}
\begin{center}
\includegraphics[scale=1]{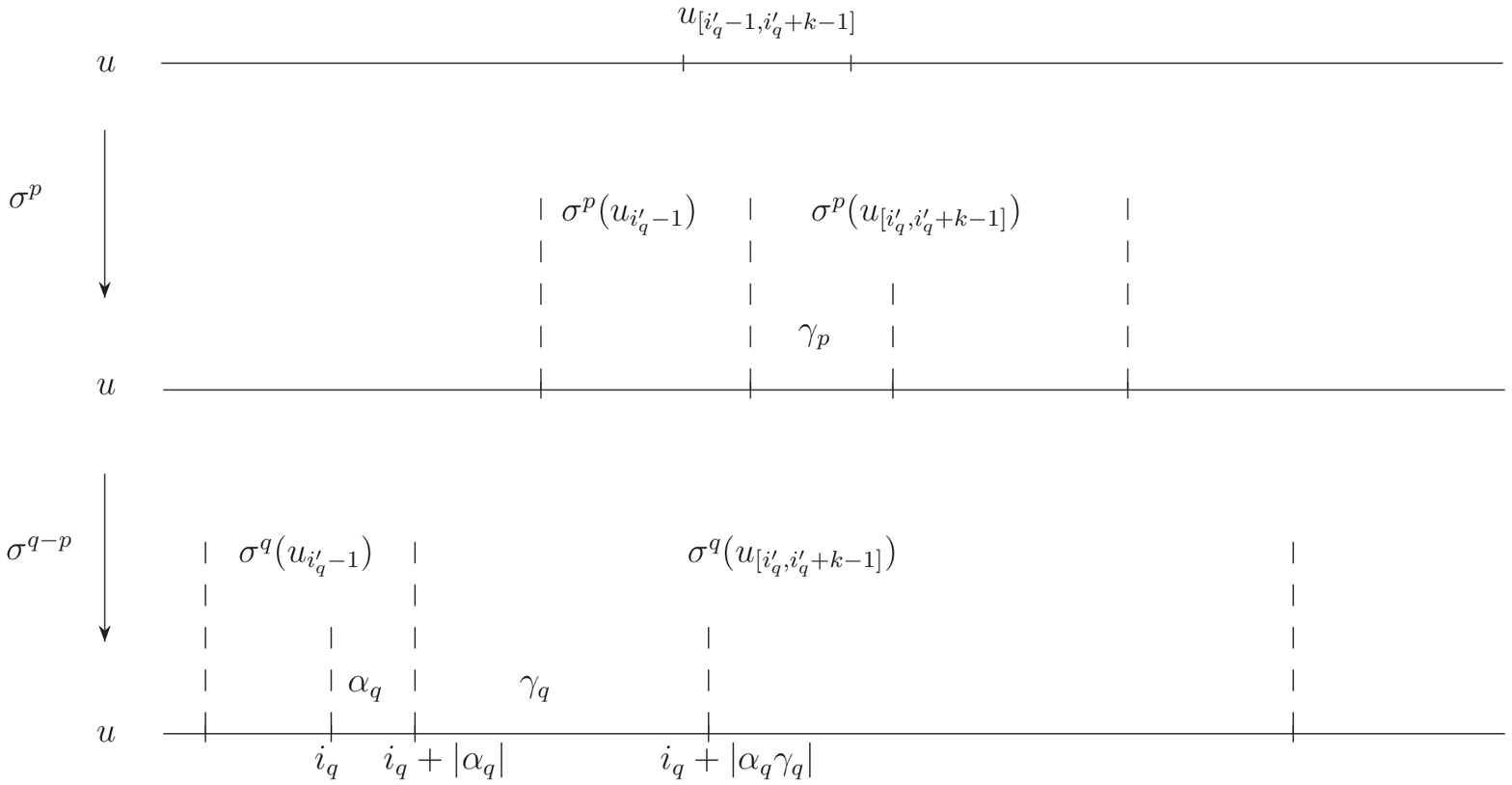}
\end{center}
\caption{}\label{pic_i_0-p-q}
\end{figure}

Then, observe how $u_{[j_q^\prime-m,j_q^\prime+h]}$ goes 
to $\sigma^q(u_{[j_q^\prime-m,j_q^\prime+h]})$ via 
$\sigma^p(u_{[j_q^\prime-m,j_q^\prime+h]})$; see Figure~\ref{pic_j_0-p-q}. 
Recalling that the natural $q$-cuttings \eqref{nat_q_cut_1} and \eqref{nat_q_cut_2} are the 
same, we can see that 
$u_{[j_q+|\alpha_q|,j_q+|\alpha_q\gamma_q|)}=\gamma_q$ has a natural $1$-cutting:
\begin{equation}\label{ntrl_1_cttng_jq_gammaq}
\Set{\sigma(u_{j^{\prime \prime}}),\sigma(u_{j^{\prime \prime}+1}), \dots, 
\sigma(u_{j^{\prime \prime}+ \vert \sigma^{q-p-1}(\gamma_p) \vert -1})},
\end{equation}
where 
$j^{\prime \prime}=
\left|\sigma^{q-p-1}
\left(u_{\left[0,\left|\sigma^p(u_{[0,j_q^\prime)})\right|-\left|\gamma_p\right|\right)}\right)\right|$. 
Remark that 
\begin{equation}\label{ancestor_of_gammaq_at_jq}
u_{[j^{\prime \prime},j^{\prime \prime}+|\sigma^{q-p-1}(\gamma_p)|)}=\sigma^{q-p-1}(\gamma_p).
\end{equation}
\begin{figure}
\begin{center}
\includegraphics[scale=1]{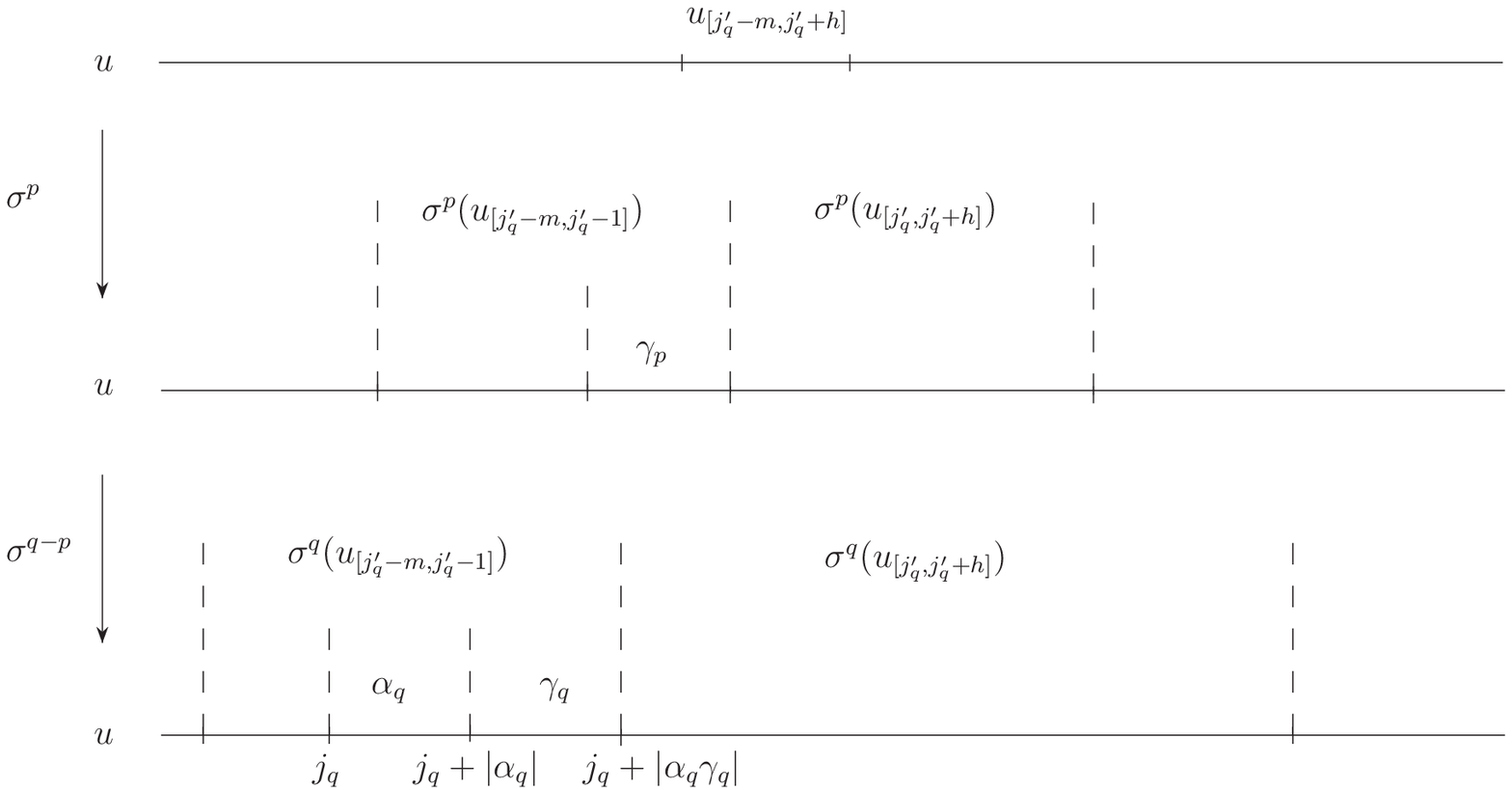}
\end{center}
\caption{}\label{pic_j_0-p-q}
\end{figure}
We are finally in a situation that 
\begin{itemize}
\item
$\alpha_q\gamma_q$ occurs at the positions $i_q \in E_1$ and $j_q \notin E_1$ in $u$;
\item
$\gamma_q$ has the same natural $1$-cutting at the positions $i_q+|\alpha_q|$ and $j_q+|\alpha_q|$; 
recall \eqref{ntrl_1_cttng_iq_gammaq} and \eqref{ntrl_1_cttng_jq_gammaq};
\item
all of the positions $i_q+|\alpha_q|,i_q+|\alpha_q\gamma_q|,j_q+|\alpha_q|$ and 
$j_q+|\alpha_q\gamma_q|$ are natural $1$-cutting points;
\item
the same natural $1$-cutting of $\gamma_q$ consists of at least $L$ words.
\end{itemize}
Actually, the second and third conditions are implied by a stronger statement that the same 
$1$-cutting of $\gamma_q$ at the positions $i_q+|\alpha_q|$ and $j_q+|\alpha_q|$ comes from the same 
word \eqref{ancestor_of_gammaq_at_iq} and \eqref{ancestor_of_gammaq_at_jq}, in other words, the 
word $\gamma_q$ has the same ancestor word \eqref{ancestor_of_gammaq_at_iq} and 
\eqref{ancestor_of_gammaq_at_jq} at the positions. We will again encounter this kind of fact in the 
proof of ``local unique composition property''~(Lemma~\ref{bil_recog_index}). 

We reach the desired positions $i,j \in \Z_+$ by executing the following procedure in this order:
\begin{enumerate}[(P.~1)]
\item
Set $\ell = i_q+|\alpha_q|$ and $m=j_q + |\alpha_q|$. 
\item\label{nearest}
Let $\ell^\prime < \ell$ and $m^\prime < m$ be natural $1$-cutting 
points which are nearest to $\ell$ and $m$ respectively. 
\item\label{case_equal}
If $\ell - \ell^\prime = m-m^\prime$, then set $\ell = \ell^\prime$ and $m = m^\prime$ and 
go back to (P.~\ref{nearest}). 
\item
In this step, we have that $\ell - \ell^\prime \ne m-m^\prime$. The desired positions $i$ and 
$j$ are determined by the facts that 
\begin{enumerate}
\item
$\ell - \ell^\prime < m - m^\prime$ $\Rightarrow$ $|\sigma(u_{[0,j)})|=\ell^\prime$ and 
$|\sigma(u_{[0,i)})|=m^\prime$;
\item
$m - m^\prime < \ell - \ell^\prime$ $\Rightarrow$ $|\sigma(u_{[0,j)})|=m^\prime$ and 
$|\sigma(u_{[0,i)})|=\ell^\prime$.
\end{enumerate}
\end{enumerate}
The loop between (P.~\ref{nearest}) and (P.~\ref{case_equal}) continues up to 
$\left\lceil |\alpha_q\gamma_q|/I_1 \right\rceil$ times. 
\end{proof}

\renewcommand{\proofname}{Proof}

\begin{lemma}\label{1st_Step}
Let $k \ge 3C^2$ be an integer. If the substitution $\sigma$ is not recognizable, then for each 
$p \in \N$ there exist integers $i_p \in E_1,j_p \notin E_1,i_p^\prime,j_p^\prime \ge 0,
h_p, \ell_p \ge 1$ and words $\alpha_p,\gamma_p^\prime \in A^\ast, \gamma_p \in A^+$ such 
that 
\begin{itemize}
\item
$u_{[i_p,i_p+\ell_p)}=u_{[j_p,j_p+\ell_p)};$
\item
$u_{[i_p,i_p+\ell_p)}$ has a natural $p$-cutting$:$
\[
\Set{\alpha_p,\sigma^p(u_{i_p^\prime}),\sigma^p(u_{i_p^\prime+1}),\dots,
\sigma^p(u_{i_p^\prime+k-1})};
\]
\item
$u_{[j_p+|\alpha_p|,j_p+\ell_p)}$ has a natural $p$-cutting$:$
\[
\Set{\gamma_p,\sigma^p(u_{j_p^\prime}),\sigma^p(u_{j_p^\prime+1}),\dots,
\sigma^p(u_{j_p^\prime+h_p-1}),\gamma_p^\prime}.
\]
\end{itemize}
\end{lemma}

\begin{proof}
Fix an integer $m_p$ with 
\[
m_p > (k+2)S_p.
\]
Since $\sigma$ is not recognizable, there exist integers $i_p \in E_1$ and $j_p \notin E_1$ 
such that $u_{[i_p,i_p+m_p)}=u_{[j_p,j_p+m_p)}$.
The choice of $m_p$ guarantees that $u_{[i_p,i_p+m_p)}$ has a natural $p$-cutting, say 
\begin{equation*}\label{original_p_cutting}
\{\alpha_p,\sigma^p(u_{i_p^\prime}),\sigma^p(u_{i_p^\prime+1}),\dots,
\sigma^p(u_{i_p^\prime+k_p-1}),\beta_p\}.
\end{equation*}
Since 
\[
k_p \ge \dfrac{m_p}{S_p}-2>k,
\]
we can see that $u_{[i_p,i_p+\ell_p)}$ has a natural $p$-cutting: 
\[
\Set{\alpha_p,\sigma^p(u_{i_p^\prime}),\sigma^p(u_{i_p^\prime+1}), \dots, 
\sigma^p(u_{i_p^\prime+k-1})},
\]
where $\ell_p=|\alpha_p\sigma^p(u_{[i_p^\prime,i_p^\prime+k)})|$. Since 
\begin{align*}
\ell_p-|\alpha_p| \ge k I_p \ge k C^{-1}\lambda^p 
\ge kC^{-2} S_p \ge 3 S_p,
\end{align*}
$u_{[j_p+|\alpha_p|,j_p+\ell_p)}$ has a natural $p$-cutting: 
\[
\Set{\gamma_p,\sigma^p(u_{j_p^\prime}),\sigma^p(u_{j_p^\prime+1}),\dots,
\sigma^p(u_{j_p^\prime+h_p-1}),\gamma_p^\prime}
\]
with $\gamma_p \ne \Lambda$. 
\end{proof}

\section{An algorithm for the unilateral non-recognizability}
Following Property~(6) in \cite{Sol}, we make a definition:
\begin{definition}
We say that the substitution $\sigma$ has the {\em local unique composition property} under an index 
$L \in \N$ if 
\begin{itemize}
\item
the substitution $\sigma$ is bilaterally recognizable under the index $L$;
\item
if $u_{[i-L,i+L)}=u_{[j-L,j+L)}, |\sigma(u_{[0,m)})| \le i < |\sigma(u_{[0,m+1)})| \textrm{ and } 
|\sigma(u_{[0,n)})| \le j < |\sigma(u_{[0,n+1)})|$ then $u_m = u_n$.
\end{itemize}
\end{definition}
See also \cite[Th\'eor\`eme~2]{MR1414542} and \cite[Theorem~11]{MR1709427}. 
None has given a computable value of an index under which the local unique composition property holds, 
which is now done by the following lemma. The lemma excluding the computability is due to 
\cite[Th\'eor\`eme~3.1 bis.]{MR1168468}, though the theorem do not mention any decidability of the 
index $L$. 

\begin{lemma}\label{bil_recog_index}
The aperiodic, primitive substitution $\sigma$ has the local unique composition property under an 
index $L_0=\{C^4(K+1)+2C^2+1\}S_{p_0}$, where 
\begin{equation}\label{p0}
p_0 = K^2 \left\{C^4(K+1)+2C^2+1 \right\} \times \sum_{n=C^2(K+1)+2}^{(K+1)C^6+2C^4+C^2+2} 
n + 1. 
\end{equation}
\end{lemma}
\begin{proof}
Put $k=C^4(K+1)+2C^2+1$. Assume that $u_{[i-L_0,i+L_0)}=u_{[j-L_0,j+L_0)}$. The integer 
$L_0$ is so large that we can choose $\Set{m_p,n_p \in \N|1 \le p \le p_0}$ so that 
\begin{itemize}
\item
$u_{[i-m_p,i+n_p)}=u_{[j-m_p,j+n_p)}$;
\item
$u_{[i-m_p,i+n_p)}$ has a natural $p$-cutting:
\[
\Set{\sigma^p(u_{i_p}),\sigma^p(u_{i_p+1}),\dots,\sigma^p(u_{i_p+k-1})}.
\]
\end{itemize}
Let 
\[
\Set{\gamma_p,\sigma^p(u_{j_p}),\sigma^p(u_{j_p+1}),\dots,\sigma^p(u_{j_p+h_p-1}),\gamma_p^\prime}
\]
be a natural $p$-cutting of $u_{[j-m_p,j+n_p)}$. Since for every integer $p$ with $1 \le p \le p_0$, 
we have 
\[
C^2(K+1)+C^{-2} = kC^{-2}-2 \le h_p \le kC^2
\]
in view of an equality: 
\[
\left|\sigma^p(u_{[i_p,i_p+k)})\right|=\left|\gamma_p \sigma^p(u_{[j_p,j_p+h_p)})\gamma_p^\prime\right|,
\]
it follows from Lemma~\ref{comlexity} that the cardinality of a set: 
\[
\Set{(u_{[i_p,i_p+k)},u_{[j_p-1,j_p+h_p]})|1 \le p \le p_0}
\]
is at most $p_0-1$. The pigeonhole principle implies that for some integers 
$p$ and $q$ with $1 \le p < q \le p_0$, 
\[
u_{[i_p,i_p+k)} = u_{[i_q,i_q+k)} \textrm{ and } u_{[j_p-1,j_p+h_p]} = u_{[j_q-1,j_q+h_q]}.
\]
Hence, $h_p = h_q$. In view of an equality: 
\[
\gamma_q\sigma^q(u_{[j_q,j_q+h_q)}) \gamma_q =\sigma^{q-p}(\gamma_p)\sigma^q(u_{[j_q,j_q+h_q)})
\sigma^{q-p}(\gamma_p^\prime),
\]
and the $(K+1)$-power freeness of the sequence $u$; recall \eqref{estimate_max_power}, we obtain that 
$\gamma_q=\sigma^{q-p}(\gamma_p)$ and $\gamma_q^\prime=\sigma^{q-p}(\gamma_p^\prime)$. Taking account 
into the positions of the natural $q$-cutting of $u_{[j-m_q,j+n_q)}$, we see that 
$u_{[i-m_q,i+n_q)}$ and $u_{[j-m_q,j+n_q)}$ have the same natural $(q-p)$-cutting, which is yielded 
by application of $\sigma^{q-p}$ to identical words:
\[
\sigma^p(u_{[i_q,i_q+k)})=\gamma_p \sigma^p(u_{[j_q,j_q+h_q)}) \gamma_p^\prime, 
\]
so that $u_{[i-m_q,i+n_q)}$ and $u_{[j-m_q,j+n_q)}$ have the same natural $1$-cutting. 
\end{proof}

To verify \cite[Th\'eor\`eme~2]{MR1414542}, B.~Moss\'e discusses such a constant $p \in \N$ that if 
$\sigma^{p-1}(a) \ne \sigma(b)^{p-1}$ and $a, b \in A$ then $\sigma^k(a) \ne \sigma^k(b)$ for all 
$k \in \Z_+$. The constant $p$ is formally obtained by setting 
\[
p=
\begin{cases}
\max_{(a,b) \in B} \min \Set{ k \in \N | \sigma^k(a) = \sigma^k(b)}+1 & \textrm{ if } 
B \ne \emptyset; \\
1 & \textrm{ otherwise},
\end{cases}
\]
where 
\[
B=\Set{(a,b) \in A \times A| a \ne b, \sigma^k(a)= \sigma^k(b) \textrm{ for some } k \in \N}.
\]
See also the proof of \cite[Theorem~4.36]{MR2041676}. As an application of Lemmas~~\ref{comlexity} and 
\ref{bil_recog_index}, we can see that 
\begin{proposition}\label{constpiscomputable}
the constant $p$ is computable. 
\end{proposition}
\begin{proof}
Put 
\begin{equation}\label{N}
N=K \left(\lfloor L_0 \lambda^{-1}(C-C^{-1}) \rfloor + 1\right) + \lfloor CL_0 \lambda^{-1} \rfloor + 1.
\end{equation}
Choose an integer $k_0$ with $k_0 > \log_\lambda (CN)$. Then, for all letters $a \in A$, 
\[
|\sigma^{k_0}(a)| \ge C^{-1}\lambda^{k_0} > N.
\]
Following \cite{Ro}, given letter $a \in A$ and integer $k$ with $k \ge k_0$, let $\Suf_N(\sigma^k(a))$ 
(resp.\ $\Pref_N(\sigma^k(a))$) denote a suffix (resp.\ prefix) of $\sigma^k(a)$ whose length is $N$. 
Fix distinct letters $a_1,a_2 \in A$. Lemma~\ref{comlexity} together with the pigeonhole principle 
allows us to find those integers $i_m < j_m$ and $k_m < \ell_m$ which belong to a closed interval 
$\left[k_0, k_0 + \lambda C^2 (\# A)^2 N \right]$ for which and for all $m=1,2$, 
\begin{itemize}
\item
$\Pref_N(\sigma^{i_m}(a_m))=\Pref_N(\sigma^{j_m}(a_m))$;
\item
$\# \Set{\Pref_N(\sigma^k(a_m))|k_0 \le k \le j_m}=j_m - k_0$;
\item
$\Suf_N(\sigma^{k_m}(a_m))=\Suf_N(\sigma^{\ell_m}(a_m))$;
\item
$\# \Set{\Suf_N(\sigma^k(a_m))|k_0 \le k \le \ell_m} = \ell_m - k_0$.
\end{itemize}
For $m=1,2$, regard 
\[
\mathcal{P}_m=\Set{\Pref_N(\sigma^k(a_m))|k \ge i_m} \textrm{ and } 
\mathcal{S}_m = \Set{\Suf_N(\sigma^k(a_m))|k \ge k_m},
\]
as sequences of words, which have periods $j_m-i_m$ and $\ell_m - k_m$, respectively. Observe that 
$\mathcal{P}_1 \cap \mathcal{P}_2 = \emptyset$ unless $\mathcal{P}_1 = \mathcal{P}_2$. This fact is 
also valid for $\mathcal{S}_m$. 

If $\mathcal{P}_1 \cap \mathcal{P}_2 = \emptyset$ or $\mathcal{S}_1 \cap \mathcal{S}_2 = \emptyset$, 
then $\sigma^k(a) \ne \sigma^k(b)$ for all $k \in \N$. Then, set $p_{a_1,a_2}=1$. 
If $\mathcal{P}_1=\mathcal{P}_2$ and $\mathcal{S}_1=\mathcal{S}_2$, then set 
\[
p_{a_1,a_2}=\max \Set{i_1,i_2,k_1,k_2,\log_\lambda(CN+C^2L_0\lambda^{-1})} + 1.
\]
Let $p$ denote $p_{a_1,a_2}$ for the simplicity of notation. 
Let us verify that if $\sigma^{p-1}(a_1) \ne \sigma^{p-1}(a_2)$ then 
$\sigma(a_1)^k \ne \sigma^k(a_2)$ for all $k \in \Z_+$. To this end, it is enough for us to consider 
only the case where $\mathcal{P}_1=\mathcal{P}_2$ and $\mathcal{S}_1=\mathcal{S}_2$. Let us see that if 
$\sigma^p(a_1) = \sigma^p(a_2)$ then $\sigma^{p-1}(a_1) = \sigma^{p-1}(a_2)$. For each $m = 1,2$, 
let $q_m < r_m$ be unique integers satisfying that 
\begin{align*}
\left\vert \sigma \left(\sigma^{p-1}(a_m)_{[1,q_m)}\right) \right\vert & \le L_0; \\
\left\vert \sigma \left(\sigma^{p-1}(a_m)_{[1,q_m]}\right) \right\vert & > L_0; \\
\left\vert \sigma \left(\sigma^{p-1}(a_m)_{[1,r_m)}\right) \right\vert & < |\sigma^p(a_m)| - L_0; \\
\left\vert \sigma \left(\sigma^{p-1}(a_m)_{[1,r_m]}\right) \right\vert & \ge |\sigma^p(a_m)| - L_0. 
\end{align*}
These inequalities together with \eqref{conseqPF} imply that for all $m=1,2$, 
\begin{align}
C^{-1}L_0 \lambda^{-1} & < q_m \le CL_0\lambda^{-1}; \label{qm} \\
C^{-1}L_0\lambda^{-1} & < |\sigma^{p-1}(a_m)| - r_m + 1 \le CL_0\lambda^{-1} + 1 \label{rm}
\end{align}
and hence 
\begin{align*}
|q_1-q_2| & < L_0\lambda^{-1}(C-C^{-1}); \\
\left||\sigma^{p-1}(a_1)| -r_1 + 1 -(|\sigma^{p-1}(a_2)| -r_2 + 1) \right| & < 
L_0\lambda^{-1}(C-C^{-1}) + 1; \\
|\sigma^{p-1}(a_m)| -q_m + 1 & \ge N +1; \\
r_m & \ge N.
\end{align*}

Lemma~\ref{bil_recog_index} allows us to know that 
\begin{equation}\label{content}
\sigma^{p-1}(a_1)_{[q_1,r_1]}=\sigma^{p-1}(a_2)_{[q_2,r_2]}.
\end{equation}
Since $\Pref_N(\sigma^p(a_1))=\Pref_N(\sigma^p(a_2))$ and 
$\Suf_N(\sigma^p(a_1))=\Suf_N(\sigma^p(a_2))$, it follows from the choice of $p$ that 
\begin{align}
\Pref_N(\sigma^{p-1}(a_1)) &= \Pref_N(\sigma^{p-1}(a_2)); \label{NPref} \\
\Suf_N(\sigma^{p-1}(a_1)) &= \Suf_N(\sigma^{p-1}(a_2)). \label{NSuf}
\end{align}
If $q_1 \ne q_2$, then the $K$-power of a word length $|q_1-q_2|$ must occur at $\min \Set{q_1, q_2}$ 
in $\sigma^{p-1}(a_m)$ if $q_m$ attains the minimum value, which is impossible in virtue of 
Lemma~\ref{uppbndforpow}. Hence, we obtain that 
\begin{equation}\label{q1=q2}
q_1=q_2,
\end{equation}
which is less than $N$ in view of 
\eqref{N} and \eqref{qm}. Similarly, we also obtain that 
\begin{align}
\left|\sigma^{p-1}(a_1)_{[r_1,|\sigma^{p-1}(a_1)|]}\right| &= 
|\sigma^{p-1}(a_1)| -r_1 + 1 \label{rm_to_right_edge} \\ 
&= |\sigma^{p-1}(a_2)| -r_2 + 1 = 
\left|\sigma^{p-1}(a_2)_{[r_2,|\sigma^{p-1}(a_2)|]}\right|, \notag 
\end{align}
which is less than $N$ in virtue of \eqref{N} and \eqref{rm}. 
Now, putting together \eqref{content}-\eqref{rm_to_right_edge}, we see that 
the words $\sigma^{p-1}(a_1)$ and $\sigma^{p-1}(a_2)$ must coincide with each other. 
\end{proof}

In order to see that a constant $L_1$ appearing in Lemma~\ref{unif_exp} is computable, we will need 
some facts about Birkhoff contraction coefficients for {\em allowable} nonnegative square matrices. We 
consult \cite[Chapter~3]{Sen} and \cite[Subsections~2.1.1 and 2.2.1]{hartfiel} for them. 
Let us consider a {\em projective metric} $d$ which is defined by for {\em positive}, row vectors 
$x,y \in \R^A$, i.e.\ all entries are positive,
\[
d(x,y)=\ln \frac{\max\limits_{a \in A} \dfrac{x_a}{y_a}}{\min\limits_{b \in A} \dfrac{x_b}{y_b}}.
\]
Suppose that a primitive matrix $M=(m_{a,b})_{a,b \in A}$ is allowable, i.e.\ every row and 
column has a positive entry. {\em Birkhoff contraction coefficient} $\tau_\mathrm{B}(M)$ is 
defined by
\[
\tau_\mathrm{B}(M)= \sup \Set{\frac{d(xM,yM)}{d(x,y)}|x,y \in \R^A \textrm{ are 
positive and linearly independent.} }.
\]
There are known properties \cite[Lemma~2.1 and Theorem~2.8]{hartfiel} that 
$0 \le \tau_\mathrm{B}(M) \le 1$ and $\tau_\mathrm{B}(MN) \le \tau_\mathrm{B}(M)\tau_\mathrm{B}(N)$ if 
$N$ is another allowable, nonnegative $A \times A$ matrix. The coefficient $\tau_\mathrm{B}(M)$ is 
computable: 
\begin{equation}\label{explicit_form}
\tau_\mathrm{B}(M)=\frac{1-\sqrt{\phi(M)}}{1+\sqrt{\phi(M)}},
\end{equation}
where 
\[
\phi(M)=\min \Set{\frac{m_{i,j}m_{k,\ell}}{m_{i,\ell}m_{k,j}}, 
\frac{m_{i,\ell}m_{k,j}}{m_{i,j}m_{k,\ell}}|\begin{pmatrix}
m_{i,j} & m_{i,\ell} \\
m_{k,j} & m_{k,\ell}
\end{pmatrix} \textrm{ is a submatrix of } M.}
\]
if $M$ is positive, and otherwise, $\phi(M) = 0$. Put 
\begin{equation}\label{1st_n0}
n_0=(\# A)^2-2(\# A)+2. 
\end{equation}
Since it follows from \cite[Theorem~2.9]{Sen} that ${M_\sigma}^n$ is positive for all integers $n$ 
with $n \ge n_0$, we have that $\phi({M_\sigma}^n) > 0$ for all such integers $n$. It follows 
from \eqref{explicit_form} that for all such integers $n$, 
\begin{equation}\label{coeff_less_one}
\tau_\mathrm{B}({M_\sigma}^n) < 1.
\end{equation}

As in Section~\ref{sec_freeness}, let $\alpha$ be a positive, left eigenvector of $M_\sigma$ 
corresponding to its Perron eigenvalue $\lambda$. Given a word $w \in A^\ast$, set 
\[
\comp(w)=(|w|_b)_{b \in A},
\]
which is viewed as a row vector in $\R^A$. Define a row vector $e_a$ in $\R^A$ by for each letter 
$b \in A$, 
\[
(e_a)_b = \begin{cases}
1 & \textrm{ if } a = b; \\
0 & \textrm{ otherwise}.
\end{cases}
\]
It is clear that $\comp \left(\sigma^n(a)\right) = e_a {M_\sigma}^n$. 
It follows from \cite[Theorem~2.9]{Sen} and \cite[Theorem~2.3 and Corollary~2.2]{hartfiel} that 
for all letter $a \in A$ and integer $n$ with $n \ge 2n_0$, 
\begin{align}
\left\Vert \frac{\comp(\sigma^n(a))}{|\sigma^n(a)|} - \frac{\alpha}{\Vert \alpha \Vert_1} \right\Vert_1 
&=
\left\Vert \frac{e_a {M_\sigma}^n}{\Vert e_a {M_\sigma}^n \Vert_1} - \frac{\alpha}{\Vert \alpha \Vert_1}
 \right\Vert_1 \notag \\ 
& \le \exp\left(d(e_a {M_\sigma}^n, \alpha)\right)-1 \notag \\
& \le \exp\left(\tau_\mathrm{B}\left({M_\sigma}^{n-n_0}\right)d(e_a{M_\sigma}^{n_0},\alpha)\right)-1 \notag \\
& \le \exp\left({\tau_\mathrm{B}({M_\sigma}^{n_0})}^{\left[\frac{n-n_0}{n_0}\right]}
\max_{a \in A}d(e_a{M_\sigma}^{n_0},\alpha) \right)-1 \notag \\
& \le \exp\left(\max_{a \in A}d(e_a{M_\sigma}^{n_0},\alpha) \right)^{%
{\tau_\mathrm{B}({M_\sigma}^{n_0})}^{\left[\frac{n}{n_0}-1\right]}}-1, \label{UB_2nd_term}
\end{align}
which monotonically decreases to zero in virtue of \eqref{coeff_less_one} as $n$ increases. As a 
trivial consequence, we obtain that 
\begin{equation}\label{indiv_UB_2nd_term}
\max_{a,b,c \in A}
\left \vert \frac{|\sigma^n(a)|_c}{|\sigma^n(a)|} - \frac{|\sigma^n(b)|_c}{|\sigma^n(b)|} \right \vert 
\le 2 \exp\left(\max_{a \in A}d(e_a{M_\sigma}^{n_0},\alpha) \right)^{%
{\tau_\mathrm{B}({M_\sigma}^{n_0})}^{\left[\frac{n}{n_0}-1\right]}}-2. 
\end{equation}

\begin{lemma}\label{unif_exp}
For every real number $\rho$ with $1 < \rho < \lambda$, there exists a computable number 
$L_1 \in \N$ so that for all integers $L \ge L_1$ and $i \ge 0$, we have an inequality$:$
\[
\vert \sigma(u_{[i,i+L)}) \vert \ge \rho \vert u_{[i,i+L)} \vert = \rho L.
\]
In particular, it holds that for all $i \in \Z_+$,
\[
|\sigma(u_{[i,i+L_1)})| \ge L_1+1.
\] 
\end{lemma}
\begin{proof}
Let $\epsilon$ denote a number satisfying that $\rho=(1-\epsilon)\lambda$. This forces that 
$0 < \epsilon < 1$. Let $n_0$ be as in \eqref{1st_n0}. Choose an integer $n$ with $n \ge n_0$ so large 
that 
\begin{equation}\label{smalle_value}
2 \exp\left(\max_{a \in A}d(e_a{M_\sigma}^{n_0},\alpha) \right)^{%
{\tau_\mathrm{B}({M_\sigma}^{n_0})}^{\left[\frac{n}{n_0}-1\right]}}-2 < \frac{\epsilon \lambda}{4 (\# A)\Vert M_\sigma \Vert_1},
\end{equation}
where $\Vert M_\sigma \Vert_1$ is a norm of $M_\sigma$ defined by 
\[
\Vert M_\sigma \Vert_1 = \max \Set{\Vert x M_\sigma \Vert_1| x \in \R^A, \Vert x \Vert_1 = 1}=
\max_{b \in A} \sum_{a \in A}(M_\sigma)_{a,b}.
\]
Consequently, the right hand side of \eqref{UB_2nd_term} is also less than \eqref{smalle_value}. 
If the integer $L$ is not less than $2S_n$, then we obtain a natural $n$-cutting: 
\[
\Set{ v, \sigma^n(u_{j+1}), \sigma^n(u_{j+2}), \dots , \sigma^n(u_{j+k}), w}
\]
of $u_{[i,i+L)}$ so that $v \prec_{\ssuf} \sigma^n(u_j)$ and $w \prec_{\spref} \sigma^n(u_{j+k+1})$. 
Put 
\[
\delta_k=1-\frac{1}{2C^2k^{-1}+1},
\]
which monotonically decreases as $k$ increases. Since $k \ge {S_n}^{-1}L -2$, we have that 
\begin{equation}\label{est_delta_by_L}
\delta_k \le \frac{2C^2}{{S_n}^{-1}L+2(C^2-1)}.
\end{equation}
Choose $L_1 \in \N$ so that for all integers $L$ with $L \ge L_1$,
\begin{equation}\label{choice_of_L1}
\max \Set{\delta_k, 2S_n L^{-1}} < \frac{\epsilon \lambda}{4 (\# A)\Vert M_\sigma \Vert_1}.
\end{equation}

Suppose that the integer $L$ is not less than $L_1$. Let $i \in \Z_+$ and $c \in A$ be arbitrary. Since
\[
\frac{|u_{[i,i+L)}|_c}{L} = \frac{\sum_{\ell = j+1}^{j+k} |\sigma^n(u_\ell)|_c}%
{\sum_{\ell = j+1}^{j+k} |\sigma^n(u_\ell)|} \cdot 
\frac{\sum_{\ell = j+1}^{j+k} |\sigma^n(u_\ell)|}{L} + \frac{|v|_c + |w|_c}{L},
\]
we obtain that 
\begin{align}
\min_{a \in A} \frac{|\sigma^n(a)|_c}{|\sigma^n(a)|} \cdot 
\frac{\sum_{\ell = j+1}^{j+k} |\sigma^n(u_\ell)|}{L}
& \le \frac{|u_{[i,i+L)}|_c}{L} \label{minmaxineq} \\ 
& \le \max_{a \in A} \frac{|\sigma^n(a)|_c}{|\sigma^n(a)|} \cdot 
\frac{\sum_{\ell = j+1}^{j+k} |\sigma^n(u_\ell)|}{L}+2S_nL^{-1}. \notag
\end{align}
However, 
\begin{equation}\label{properlength}
1-\delta_k=\frac{1}{2C^2k^{-1}+1} \le \frac{\sum_{\ell = j+1}^{j+k} |\sigma^n(u_\ell)|}{L} \le 1. 
\end{equation}
Let $b \in A$ be arbitrary. Putting \eqref{minmaxineq} and \eqref{properlength} together, we obtain that 
\[
(1-\delta_k)\min_{a \in A} \frac{|\sigma^n(a)|_c}{|\sigma^n(a)|} - \frac{|\sigma^n(b)|_c}{|\sigma^n(b)|} 
\le \frac{|u_{[i,i+L)}|_c}{L}  - \frac{|\sigma^n(b)|_c}{|\sigma^n(b)|} \le 
\max_{a \in A} \frac{|\sigma^n(a)|_c}{|\sigma^n(a)|} - \frac{|\sigma^n(b)|_c}{|\sigma^n(b)|} + 2S_nL^{-1},
\]
and hence,
\[
\left\vert \frac{|u_{[i,i+L)}|_c}{L} - \frac{|\sigma^n(b)|_c}{|\sigma^n(b)|} \right\vert \le 
\max_{a \in A} \left\vert \frac{|\sigma^n(a)|_c}{|\sigma^n(a)|} - 
\frac{|\sigma^n(b)|_c}{|\sigma^n(b)|} \right\vert + \max \Set{\delta_k, 2S_nL^{-1}}. 
\]
Using \eqref{indiv_UB_2nd_term}, \eqref{smalle_value} and \eqref{choice_of_L1}, we hence obtain that 
\begin{align*}
\left\Vert \frac{\comp(u_{[i,i+L)})}{L} - \frac{\comp(\sigma^n(b))}{|\sigma^n(b)|}\right\Vert_1 
& \le \sum_{c \in A} 
\max_{a \in A} \left\vert \frac{|\sigma^n(a)|_c}{|\sigma^n(a)|} - 
\frac{|\sigma^n(b)|_c}{|\sigma^n(b)|} \right\vert + (\# A) \max \Set{\delta_k, 2S_nL^{-1}} \\
& \le \frac{\epsilon \lambda}{2 \Vert M_\sigma \Vert_1}.
\end{align*}
We finally obtain that 
\begin{align*}
\left\Vert \frac{\comp(u_{[i,i+L)})}{L} - \frac{\alpha}{\Vert \alpha \Vert_1}\right\Vert_1 & \le 
\left\Vert \frac{\comp(u_{[i,i+L)})}{L} - \frac{\comp(\sigma^n(b))}{|\sigma^n(b)|}\right\Vert_1 + 
\left\Vert \frac{\comp(\sigma^n(b))}{|\sigma^n(b)|} - \frac{\alpha}{\Vert \alpha \Vert_1} 
\right\Vert_1 \\ 
& \le \frac{\epsilon \lambda}{\Vert M_\sigma \Vert_1}
\end{align*}
and hence,
\begin{align*}
\vert \sigma(u_{[i,i+L)}) \vert &= L \left\Vert \frac{\comp (u_{[i,i+L)})}{L}M_\sigma \right\Vert_1 \\
& \ge 
L\left\vert \left\Vert \frac{\alpha}{\Vert \alpha \Vert_1}M_\sigma \right\Vert_1 - 
\left\Vert \left(\frac{\comp (u_{[i,i+L)})}{L}-\frac{\alpha}{\Vert \alpha \Vert_1}\right)M_\sigma 
\right\Vert_1\right\vert \\
& \ge L(1-\epsilon)\lambda = \rho L.
\end{align*}

It is now clear that $L_1$ is computable, because in virtue of \eqref{est_delta_by_L} and 
\eqref{choice_of_L1} it is sufficient to choose $L_1$ so that 
\[
\max \Set{\frac{2C^2}{{S_n}^{-1}L_1+2(C^2-1)}, 2S_n {L_1}^{-1}} < \frac{\epsilon \lambda}
{4 (\# A)\Vert M_\sigma \Vert_1}.
\]
This completes the proof. 
\end{proof}
Fix $1 < \rho < \lambda$ and $L_1 \in \N$ as in Lemma~\ref{unif_exp}. 
Fix an integer $N$ greater than 
\[
\max\Set{S_1 (L_1+1) -1, \frac{1+L_0({I_1}^{-1}+1)}{\rho-1}}.
\]
Set 
\[
V^\ast=\Set{(xac,ybc) \in \L_{N+L_1+1}(u) \times \L_{N+L_1+1}(u) | a,b \in A \ (a \ne b), \ c 
\in \L_{L_1}(u)},
\]
which is nonempty, because the sequence $u$ over the finite alphabet $A$ is assumed to be 
aperiodic; see \cite[Theorem~2.11]{CH} and \cite[Proposition V.18]{MR924156}. Define an equivalence 
relation $\sim$ on $V^\ast$ so that $(v,w) \sim (v^\prime,w^\prime)$ if and only if 
$(v,w) = (v^\prime,w^\prime)$ or $(v,w) \sim (w^\prime,v^\prime)$. Set $V=V^\ast/\sim$. Let 
$[v,w]$ denote the equivalence class of a given element $(v,w)$ of $V^\ast$. 
Consider a directed,finite graph $G$ with vertex set $V$ and edge set $E$. 
The edge set $E$ is defined by declaring that there exists an edge from a vertex $v$ to a vertex 
$v^\prime$ if and only if there exist word $w \in A^+$, representatives $(s,t)$ and 
$(s^\prime,t^\prime)$ of the equivalence classes $v$ and $v^\prime$, respectively, such that 
$s^\prime w \prec_{\suf} \sigma(s)$ and $t^\prime w \prec_{\suf} \sigma(t)$. 
\begin{remark}\label{edgesarefunction}
The number of edges leaving a given vertex is at most one. 
\end{remark}

\begin{definition}
We say that a vertex of the directed, finite graph $G$ {\em generates a gap of natural $1$-cutting 
points} if there exist letters $\alpha, \beta$, words $\gamma, \delta$ of the same length and 
representative $(v,w)$ of the vertex so that 
\begin{itemize}
\item
$\sigma(\alpha) \prec_{\ssuf} \sigma(\beta)$;
\item
$\sigma(\gamma_i)=\sigma(\delta_i)$ for every integer $i$ with $1 \le i \le |\gamma|$;
\item
$\alpha \gamma \prec_{\suf} \sigma(v)$ and $\beta \delta \prec_{\suf} \sigma(w)$.
\end{itemize}
\end{definition}
We define a unilateral subshift:
\[
X_\sigma=\Set{x=(x_i)_{i \in \Z_+}|x_{[i,j]} \in \L(u) \textrm{ for all }i,j \in \Z_+}. 
\]
In other words, the unilateral subshift $X_\sigma$ is generated by the language of the sequence $u$. 
\begin{theorem}
The following are equivalent$:$
\begin{enumerate}
\item\label{non-recogforcycle}
the substitution $\sigma$ is not unilaterally recognizable$;$
\item\label{thecycle}
the directed, finite graph $G$ has a cycle including a vertex generating a gap of natural 
$1$-cutting points.
\end{enumerate}
\end{theorem}

\begin{proof}
\eqref{thecycle} $\Rightarrow$ \eqref{non-recogforcycle}: Assume that the directed, finite graph $G$ 
includes a cycle: 
\[
\Set{[x_ia_ic_i,y_ib_ic_i] \in V|0 \le i \le \ell}
\]
of length $\ell$ so that 
\begin{itemize}
\item
the vertex $[x_0a_0c_0,y_0b_0c_0]$ generates a gap of natural $1$-cutting points;
\item
$x_\ell a_\ell c_\ell = x_0 a_0 c_0$ and $y_\ell b_\ell c_\ell = y_0 b_0 c_0$;
\item
for every integer $i$ with $0 \le i < \ell$, there exists $w_{i+1} \in A^+$ satisfying that 
\[
x_{i+1}a_{i+1}c_{i+1}w_{i+1} \prec_{\suf} \sigma(x_ia_ic_i) \textrm{ and } 
y_{i+1}b_{i+1}c_{i+1}w_{i+1} \prec_{\suf} \sigma(y_ib_ic_i).
\]
\end{itemize}
For every integer $i$ with $i > \ell$, put $w_i=w_{(i \mod \ell) + 1}$. 
It is straightforward to see that for every $k \in \N$, 
\begin{align*}
x_0a_0c_0 w_{k\ell}\sigma(w_{k\ell-1}) \sigma^2(w_{k\ell-2}) \dots \sigma^{k\ell-2}(w_2) 
\sigma^{k\ell-1}(w_1) & \prec_{\suf} \sigma^{k\ell}(x_0a_0c_0); \\
y_0b_0c_0 w_{k\ell}\sigma(w_{k\ell-1}) \sigma^2(w_{k\ell-2}) \dots \sigma^{k\ell-2}(w_2) 
\sigma^{k\ell-1}(w_1) & \prec_{\suf} \sigma^{k\ell}(y_0b_0c_0).
\end{align*}
Hence, Condition~\eqref{charac} in Theorem~\ref{necsuffcondunilrecog} is satisfied. 

\eqref{non-recogforcycle} $\Rightarrow$ \eqref{thecycle}: 
Assume that $\sigma$ is not unilaterally recognizable. We shall see that the directed, finite graph 
$G$ has a vertex which generates a gap of natural $1$-cutting points. Using the pigeonhole principle 
together with Theorem~\ref{necsuffcondunilrecog} and the uniform recurrence of the sequence $u$, we 
can find $x,y \in \L_N(u)$, $a,b \in A$ and $z,w \in X_\sigma$ so that 
\begin{itemize}
\item
$xaz,ybw \in X_\sigma$;
\item
$\sigma(a) \prec_{\ssuf} \sigma(b)$;
\item
$\sigma(z_i)=\sigma(w_i)$ for all $i \in \Z_+$.
\end{itemize}
Lemma~\ref{bil_recog_index} allows us to find $\ell \in \Z_+$ such that 
\begin{itemize}
\item
$z_{\ell-1} \ne w_{\ell-1}$;
\item
$z_{[\ell,+\infty)}=w_{[\ell,+\infty)}$;
\item 
$|\sigma(z_{[0,\ell)})| \le L_0$;
\item
$|\sigma(w_{[0,\ell)})\vert \le L_0$,
\end{itemize}
where we use a convention that $z_{-1}=a$ and $w_{-1}=b$. 

By the pigeonhole principle again, using the hypothesis that the sequence $u$ is assumed to be 
aperiodic, we can find $e,f \in X_\sigma$ for which 
\begin{equation}\label{minimal_cover}
xaz \prec_{\suf} \sigma(e), \ xaz \nprec_{\suf} \sigma(e_{[1,+\infty)}), \ ybw \prec_{\suf} \sigma(f) 
\textrm{ and } ybw \nprec_{\suf} \sigma(f_{[1,+\infty)}). 
\end{equation}
In view of Lemma~\ref{bil_recog_index} again, there exist $m,n \in \N$ such that 
\begin{itemize}
\item
$\alpha:=e_{m-1} \ne f_{n-1}=:\beta$;
\item
$\zeta:=e_{[m,+\infty)}=f_{[n,+\infty)}$. 
\end{itemize}
Recall that $z_{\ell-1} \ne w_{\ell-1}$. There exists a prefix $s \in A^\ast$ of 
$z_{[\ell,+\infty)}=w_{[\ell,+\infty)}$ such that 
\begin{equation}\label{t}
|s| \le L_0, \ xaz_{[0,\ell)}s \prec_{\suf} \sigma(e_{[0,m)}) \textrm{ and } 
ybw_{[0,\ell)}s \prec_{\suf} \sigma(f_{[0,n)}).
\end{equation}
Since 
\[
m S_1 \ge |\sigma(e_{[0,m)})| \ge |xa| = N+1 \ge S_1 (L_1+1),
\]
we obtain that $m-1 \ge L_1$. This together with Lemma~\ref{unif_exp} implies that 
\[
\rho(m-1) \le |\sigma(e_{[1,m)})| < |xaz_{[0,\ell)}s| \le N + 1 + L_0({I_1}^{-1}+1) < \rho N,
\]
where the second inequality follows from the second property of \eqref{minimal_cover}, so that 
$m-1 < N$. 
Similarly, we obtain that $n-1 < N$. These facts allow us to find words $\chi,\tau \in \L_N(u)$ so that 
\begin{itemize}
\item
$e_{[0,m-1)} \prec_{\suf} \chi$;
\item
$\chi \alpha \zeta \in X_\sigma$;
\item
$f_{[0,n-1)} \prec_{\suf} \tau$;
\item
$\tau\beta \zeta \in X_\sigma$.
\end{itemize}
Consequently, we obtain that 
\begin{itemize}
\item
$\sigma(e) \prec_{\suf} \sigma(\chi)\sigma(e_{[m-1,+\infty)}) = \sigma(\chi \alpha \zeta)$;
\item
$\sigma(f) \prec_{\suf} \sigma(\tau)\sigma(f_{[n-1,+\infty)}) = \sigma(\tau \beta \zeta)$.
\end{itemize}
This together with \eqref{minimal_cover} shows that 
\begin{itemize}
\item
$xaz \prec_{\suf} \sigma(\chi \alpha \zeta)$;
\item
$ybw \prec_{\suf} \sigma(\tau \beta \zeta)$.
\end{itemize}
Since in virtue of \eqref{t} we know that $xaz_{[0,\ell)}s \prec_{\suf} \sigma(\chi \alpha)$, letting 
$\gamma = \zeta_{[0,L_1)}$, we obtain that 
$xaz_{[0,\ell)}s \sigma(\gamma) \prec_{\suf} \sigma(\chi \alpha \gamma)$, and also that 
$ybw_{[0,\ell)}s\sigma(\gamma) \prec_{\suf} \sigma(\tau \beta \gamma)$. 
We have obtained a vertex $[\chi \alpha \gamma, \tau \beta \gamma] \in V$ which generates a 
gap of natural $1$-cutting points. 

Now, apply the procedure to $(\chi \alpha \zeta, \tau \beta \zeta)$, which has been 
applied to $(xaz,ybw)$ for obtaining $(\chi \alpha \zeta, \tau \beta \zeta)$. 
It yields another vertex where an edge 
leaves for $(\chi \alpha \gamma, \tau \beta \gamma)$. Applying the procedure inductively 
yields an infinite path in the directed, finite graph, which results in a cycle in virtue of 
Remark~\ref{edgesarefunction}. 
\end{proof}


\paragraph{{\bf Acknowledgments}}
The first revision was undertaken during the third author's visit in 2014 at Laboratoire 
Ami\'enois de Math\'ematiques Fondamentales et Appliqu\'ees, CNRS-UMR 6140, Universit\'e de 
Picardie Jules Verne. He is grateful for hospitality and kind support extended to him by the 
institution. The second revision was done in 2015 after his visit at Institute of 
Mathematics, University of Tsukuba. This revision was done in 2017 after his visit at Centre 
International de Rencontres Math\'ematiques. He is also grateful for their hospitality and discussion. 
This work was partially supported by JSPS KAKENHI Grant Number 17K05159.

\end{document}